\documentclass[12pt]{amsart}

\usepackage{amssymb,latexsym}

\usepackage{enumerate}

\makeatletter

\@namedef{subjclassname@2010}{

  \textup{2010} Mathematics Subject Classification}

\makeatother
\newtheorem{thm}{Theorem}[section]
\newtheorem{Thm}{Theorem}
\newtheorem{Con}{Conjecture}

\newtheorem{lem}[thm]{Lemma}
\newtheorem{pro}[thm]{Proposition}
\theoremstyle{definition}

\numberwithin{equation}{section}

\newcommand{\ex}{\mathbb{E}}
\newcommand{\re}{\textup{Re}}
\newcommand{\im}{\textup{Im}}
\newcommand{\ns}{\widetilde{S_{\chi}}}
\newcommand{\cha}{\Phi_{\chi}}
\begin{document}

\baselineskip=17pt

\title[]{The distribution of short character sums}

\author[Youness Lamzouri]{Youness Lamzouri}

\address{Department of Mathematics, University of Illinois at Urbana-Champaign,
1409 W. Green Street,
Urbana, IL, 61801
USA}

\email{lamzouri@math.uiuc.edu}

\date{}

\begin{abstract}
Let $\chi$ be a non-real Dirichlet character modulo a prime $q$. In this paper we prove that the distribution of the short character sum $S_{\chi,H}(x)=\sum_{x< n\leq x+H} \chi(n)$, as $x$ runs over the positive integers below $q$, converges to a two-dimensional Gaussian distribution on the complex plane, provided that $\log H=o(\log q)$ and $H\to\infty$ as $q\to\infty$. Furthermore, we use a method of Selberg to give an upper bound on the rate of convergence.
\end{abstract}


\thanks{The author is supported by a Postdoctoral Fellowship from the Natural Sciences and Engineering Research Council of Canada.}

\maketitle

\section{Introduction}
Understanding the behavior of character sums is one of the central problems in analytic number theory. Let $\chi$ be a non-principal Dirichlet character modulo a large prime $q$, and define the short character sum
$$ S_{\chi, H}(x):=\sum_{x<n\leq x+H}\chi(n),$$
where $H=H(q)\leq q$. When $\chi= \left(\frac{\cdot}{q}\right)$ is the Legendre symbol modulo $q$, Davenport and Erd\"os \cite{DE} studied the distribution of $S_{\chi,H}(x)$ as $x$ runs over the positive integers below $q$, and proved that it tends to a normal distribution of mean zero and variance $H$, if $\log H=o(\log q)$ and  $H\to \infty$ as $q\to \infty$ (this restriction on $H$ makes it possible to compute the moments of $S_{\chi,H}(x)$). More precisely they showed that in this range
\begin{equation}
 \frac{1}{q}\left|\left\{0\leq x\leq q-1: \frac{S_{\chi,H}(x)}{\sqrt{H}}\leq \lambda\right\}\right|\to \frac{1}{2\pi}\int_{-\infty}^{\lambda}e^{-x^2/2}dx,
\end{equation}
as $q,H\to \infty$. However, their method does not give information on the rate of this convergence.

An analogous result to (1.1) in the case of sums of the M\"obius function in short intervals have been obtained by Ng \cite{Ng} conditionally on a variant of the Hardy-Littlewood prime $k$-tuples conjecture adapted to the case of the M\"obius function. His result states that, under this assumption, the distribution of the sum $\sum_{x<n\leq x+H}\mu(n)$ as $x$ varies over the integers below $N$ converges to a Gaussian of mean $0$ and variance $6H/\pi^2$, if $\log H=o(\log N)$ and $H \to\infty$ as $N\to \infty.$

Recently in \cite{MZ}, Mak and Zaharescu generalized  Davenport and Erd\"os approach by investigating the distribution of more general short exponential sums. Their results imply  that for any non-real character $\chi$ modulo $q$, both $\re{S_{\chi,H}(x)}$ and $\im{S_{\chi,H}(x)}$ have a limiting Gaussian distribution. However, their method does not seem to handle the joint distribution of $\re{S_{\chi, H}(x)}$ and $\im{S_{\chi, H}(x)}$.

In this paper we extend both the results of Davenport-Erd\"os and Mak-Zaharescu, namely by proving that $S_{\chi,H}(x)$ has a two-dimensional Gaussian distribution on the complex plane. Our approach relies on a method of Selberg originally applied to  prove log-normality for the values of $\zeta(1/2+it)$ (see Tsang's thesis \cite{Ts} and Selberg \cite{Se}), where $\zeta(s)$ is the Riemann zeta function.
\begin{Thm} Let $\chi$ be a non-real character modulo a large prime $q$, and let $\mathcal{R}$ be a closed rectangle in the complex plane with edges parallel to the axes. If $\log H=o(\log q)$ and $H\to\infty$ as $q\to\infty$ then
\begin{align*}
\frac{1}{q}\left|\left\{0\leq x\leq q-1: \frac{S_{\chi,H}(x)}{\sqrt{H/2}}\in \mathcal{R}\right\}\right|&= \frac{1}{2\pi}\iint_{\mathcal{R}}\exp\left(-\frac{x^2+y^2}{2}\right)dxdy\\
&+ O\left(\left(\mu_2(\mathcal{R})+1\right)\left(H^{-1/4}+ \sqrt{\frac{\log H}{\log q}}\right)\right),
\end{align*}
where $\mu_2$ stands for the two-dimensional Lebesgue measure.
\end{Thm}
We also note that our method furnishes the same bound (as in Theorem 1) for the rate of convergence of the distribution of $S_{\chi,H}(x)/\sqrt{H}$ to the standard Gaussian in (1.1), in the case where $\chi$ is the Legendre symbol modulo $q$. Moreover, it is possible to obtain an analogous result to Theorem 1 for more general short exponential sums if one combines our approach with the work of Mak and Zaharescu in \cite{MZ}.

Remark that if $\chi$ is non-real we need to normalize $S_{\chi,H}(x)$ by $\sqrt{H/2}$ while if $\chi$ is real then the normalization factor is $\sqrt{H}$ (see (1.1)). Indeed, one can show that in the latter case
$$ \frac{1}{q}\sum_{x=0}^{q-1}S_{\chi,H}(x)= o(1) \quad \text{ and } \quad \frac{1}{q}\sum_{x=0}^{q-1}\left|S_{\chi,H}(x)\right|^2=(1+o(1)) H,$$
so that $H$ is the variance of $S_{\chi,H}(x)$ in this case. On the other hand, if $\chi$ is non-real then it follows from Proposition 2.2 below that
$$ \frac{1}{q}\sum_{x=0}^{q-1}\left|\re S_{\chi,H}(x)\right|^2 = \left(\frac{1}{2}+o(1)\right)\frac{1}{q}\sum_{x=0}^{q-1}\left|S_{\chi,H}(x)\right|^2\sim \frac{H}{2}$$ and that the same estimate holds for $\im S_{\chi,H}(x)$.

In analogy with multiplicative functions, Chatterjee and Soundararajan \cite{CS} studied the distribution of random multiplicative functions $X(n)$ in short intervals. In this probabilistic model, the values of the multiplicative function at the primes $X(p)$ are independent random variables which take the values $-1$ and $1$ with probability $1/2$, and the values at all natural numbers are built out of the values at primes by the multiplicative property. Using Stein's method for normal approximation they proved that $\sum_{N< n\leq N+H}X(n)$ is approximately Gaussian if $N^{1/5}\log N\ll H=o(N/\log N)$. This random model was originally introduced in order to understand the distribution of partial sums of the M\"obius function, which explains the restriction on the lower bound of $H$ (this is connected to the problem of the existence of square-free integers in short intervals). In our case, we are concerned with Dirichlet characters, and so the corresponding probabilistic model involves random completely multiplicative functions. Therefore in this case,  Chatterjee and Soundararajan result should hold when $H=o(N/\log N)$ and $H\to\infty$ as $N\to \infty$. Motivated by this argument we formulate the following conjecture:
\begin{Con} Let $\chi$ be a non-principal character modulo a prime $q$. Suppose that $H=H(q)=o(q/\log q)$ and $H\to\infty$ as $q\to\infty$. Then, when $\chi$ is real, the distribution of $\frac{S_{\chi, H}(x)}{\sqrt{H}}$ for $0\leq x\leq q-1$ tends to a  Gaussian with mean $0$ and variance $1$. Moreover, when $\chi$ is non-real then $\frac{S_{\chi, H}(x)}{\sqrt{H/2}}$ has a limiting two-dimensional standard  Gaussian distribution on the complex plane.
\end{Con}

\section{Moments of short character sums}

Let $\chi$ be a non-real character modulo $q$ and $1\leq H\leq q$ be a positive integer. In order to investigate the joint distribution of $\re S_{\chi,H}(x)$ and $\im S_{\chi,H}(x)$  we shall compute the moments
$$M(r,s):= \frac{1}{q}\sum_{x=0}^{q-1}\left(\re S_{\chi,H}(x)\right)^r\left(\im S_{\chi,H}(x)\right)^s,$$
for non-negative integers $r,s$.
Mak and Zaharescu \cite{MZ} had previously computed the moments of $\re S_{\chi,H}(x)$ (and also those of $\im S_{\chi,H}(x)$) and proved that they are close to the moments of a Gaussian. In our case applying their method leads to weaker error terms. Instead,
our strategy is to introduce a more adequate probabilistic model for the character sum $S_{\chi,H}(x)$. Let $X_1,\dots, X_H$ be independent random variables uniformly distributed on the unit circle, and define
$$ Z_H:=X_1+\cdots +X_H.$$
Using the Weil bound for character sums (see Weil \cite{We}) we will prove that the moments $M(r,s)$ are very close to the corresponding moments of the random variable $Z_H$, that is to say the expectation $\ex\big((\re Z_H)^r(\im Z_H)^s\big).$ (Throughout the paper $\ex(\cdot)$ stands for the expectation of the random variable in brackets). First we require the following lemma.
\begin{lem} Let $r, s$ be non-negative integers. If $r+s$ is odd then $$\ex\big((\re Z_H)^r(\im Z_H)^s\big)=0.$$
On the other hand, if $r+s=2m$ is even then
$$ \ex\big((\re Z_H)^r(\im Z_H)^s\big)= B_m(H)\sum_{\substack{0\leq j\leq r, 0\leq k\leq s\\
j+k=m }}\frac{1}{2^r(2i)^s}\binom{r}{j}\binom{s}{k}(-1)^{s-k},$$
where $B_m(H)$ is the number of positive integers $1\leq y_1,\dots, y_m,z_1,\dots,z_m \leq H$ such that $\{y_1,\dots,y_m\}=\{z_1,\dots,z_m\}$.
\end{lem}
\begin{proof} We have
\begin{equation}
\begin{aligned}
\ex\big((\re Z_H)^r(\im Z_H)^s\big)&= \ex\left(\left(\frac{Z_H+\overline{Z_H}}{2}\right)^r
\left(\frac{Z_H-\overline{Z_H}}{2i}\right)^s\right)\\
&=\ex\left(\sum_{j=0}^r 2^{-r}\binom{r}{j}Z_H^j\overline{Z_H}^{r-j}\sum_{k=0}^s (2i)^{-s}\binom{s}{k}Z_H^k(-\overline{Z_H})^{s-k}\right)\\
&= \sum_{j=0}^r\sum_{k=0}^s \frac{1}{2^r(2i)^s}\binom{r}{j} \binom{s}{k}(-1)^{s-k}\ex\left(Z_H^{j+k}\overline{Z_H}^{r+s-(j+k)}\right).
\end{aligned}
\end{equation}
Moreover, notice that
\begin{equation*}
\begin{aligned}
\ex\left(Z_H^{k}\overline{Z_H}^{l}\right)&= \ex\left(\left(\sum_{1\leq n\leq H}X_n\right)^k\left(\sum_{1\leq m\leq H}\overline{X_m}\right)^l\right)\\
&=\sum_{1\leq n_1,\dots, n_k\leq H}\sum_{1\leq m_1,\dots,m_l\leq H}\ex\left(X_{n_1}\cdots X_{n_k}\overline{X_{m_1}\cdots X_{m_l}}\right).
\end{aligned}
\end{equation*}
Hence, we find
$$ \ex\left(Z_H^{k}\overline{Z_H}^{l}\right)=\begin{cases} B_k(H) & \text{ if } k=l\\
0 & \text{ otherwise}.\\
\end{cases}
$$
The result follows upon inserting this estimate in (2.1).
\end{proof}
Using this lemma we establish the following result:
\begin{pro} For any non-negative integers $r,s$ such that $H^{r+s}\leq q^{1/2}$ we have
$$M(r,s)=\ex\big((\re Z_H)^r(\im Z_H)^s\big)+ O\left(H^{r+s}(r+s)q^{-1/2}\right).$$
\end{pro}
\begin{proof} First, similarly to (2.1) we derive
\begin{equation}
\begin{aligned}
M(r,s)&= \frac{1}{q }\sum_{x=0}^{q-1}\left(\frac{S_{\chi,H}(x)+\overline{S_{\chi,H}(x)}}{2}\right)^r
\left(\frac{S_{\chi,H}(x)-\overline{S_{\chi,H}(x)}}{2i}\right)^s\\
&=\frac{1}{q }\sum_{x=0}^{q-1}\sum_{j=0}^r 2^{-r}\binom{r}{j}S_{\chi,H}(x)^j\overline{S_{\chi,H}(x)}^{r-j}\sum_{k=0}^s (2i)^{-s}\binom{s}{k}S_{\chi,H}(x)^k(-\overline{S_{\chi,H}(x)})^{s-k}\\
&= \sum_{j=0}^r\sum_{k=0}^s \frac{1}{2^r(2i)^s}\binom{r}{j} \binom{s}{k}(-1)^{s-k}I\big(j+k, r+s-(j+k)\big),
\end{aligned}
\end{equation}
where
$$ I(k,l)=\frac{1}{q}\sum_{x=0}^{q-1}S_{\chi,H}(x)^k \overline{S_{\chi,H}(x)}^l.$$
Furthermore, we have
\begin{align*}
I(k,l)&=\frac{1}{q}\sum_{x=0}^{q-1}\left(\sum_{x<n\leq x+H}\chi(n)\right)^k\left(\sum_{x< m\leq x+H}\overline{\chi(m)}\right)^l\\
&= \frac{1}{q}\sum_{x=0}^{q-1}\sum_{x<n_1,\dots,n_k\leq x+H}\sum_{x< m_1,\dots,m_l\leq x+H} \chi(n_1\cdots n_k)\overline{\chi(m_1\cdots m_l)}.
\end{align*}
Writing $y_j=n_j-x$ for $1\leq j\leq k$ and $z_t=m_t-x$ for $1\leq t\leq l$, we deduce that
\begin{equation}
I(k,l)= \sum_{1\leq y_1,\dots, y_k\leq H}\sum_{1\leq z_1,\dots, z_l\leq H} \frac{1}{q}\sum_{x=0}^{q-1}\chi\Big((x+y_1)\cdots (x+y_k)\Big)\overline{\chi}\Big((x+z_1)\cdots (x+z_l)\Big).
\end{equation}
We split the set of positive integers $1\leq y_1,\dots, y_k,z_1,\dots,z_l \leq H$ into diagonal terms $\{y_1,\dots, y_k\}=\{z_1,\dots,z_l\}$ (which will give the main contribution to $I(k,l)$), and off-diagonal terms $\{y_1,\dots, y_k\}\neq\{z_1,\dots,z_l\}$. To bound the contribution of the off-diagonal terms, we use the Weil bound for character sums (more precisely Corollary 11.24 of Iwaniec-Kowalski \cite{IK}) which yields in this case
\begin{equation}\left|\sum_{x=0}^{q-1}\chi\Big((x+y_1)\cdots (x+y_k)\Big)\overline{\chi}\Big((x+z_1)\cdots (x+z_l)\Big)\right|\leq (k+l)\sqrt{q}.
\end{equation}
On the other hand, if $\{y_1,\dots, y_k\}=\{z_1,\dots,z_l\}$ then $k=l$ and
\begin{equation}
\sum_{x=0}^{q-1}\chi\Big((x+y_1)\cdots (x+y_k)\Big)\overline{\chi}\Big((x+z_1)\cdots (x+z_l)\Big)=q+O(k),
\end{equation}
since $\chi(x+y)=0$ if $x\equiv -y\bmod q$.

Therefore, in the case where $k\neq l$, all the terms in the sum on the RHS of (2.3) are off-diagonal, which in view of (2.4) gives
$$
 |I(k,l)|\leq (k+l)H^{k+l}q^{-1/2}.
$$
On the other hand, if $k=l$ then using (2.4) and (2.5) we get
$$ I(k,l)= B_k(H) \left(1+O\left(\frac{k}{q}\right)\right) + O\left(k H^{2k}q^{-1/2}\right)=B_k(H)+ O\left(k H^{2k}q^{-1/2}\right),$$
since $B_k(H)\leq H^{2k}.$
Thus, inserting these estimates in (2.2) and appealing to Lemma 2.1 we get
$$M(r,s)=\ex\big((\re Z_H)^r(\im Z_H)^s\big)+ O\left(H^{r+s}(r+s)q^{-1/2}\right),$$
as desired.
\end{proof}
\section{An asymptotic formula for the two-dimensional characteristic function of $S_{\chi,H}(x)$}

To lighten the notation, we shall define the normalized short character sum by
$$\ns(x):= \frac{S_{\chi,H}(x)}{\sqrt{H/2}}.$$ Let $\cha$ be the characteristic function of the joint distribution of $\re\ns$ and $\im\ns$, which is defined by
$$ \cha(u,v)=\frac{1}{q}\sum_{x=0}^{q-1}\exp\left(iu\re\ns(x)+iv\im\ns(x)\right).$$

The purpose of this section is to establish the following theorem, which shows that $\Phi_{\chi}(u,v)$ is very close  to the characteristic function of a two-dimensional standard Gaussian distribution. This will be the main ingredient of the proof of Theorem 1.
\begin{thm} Let $q$ be a large prime and  $N$ be a positive integer such that $N\leq \log q/(20\log H).$ Then, for any real numbers $u,v$ such that $|u|,|v|\leq H^{1/4}$ we have
\begin{equation*}
\begin{aligned}
\cha(u,v)=& \exp\left(-\frac{u^2+v^2}{2}\right)\left(1+O\left(\frac{u^4+v^4}{H}\right)\right)\\
 &+ O\left(\frac{(2u^2)^N+(2v^2)^N}{N!}+ \frac{(2uv)^{2N}}{(2N)!}+ q^{-1/4}(1+u^{2N})(1+v^{2N})\right).
 \end{aligned}
\end{equation*}

\end{thm}
In order to prove this result, we shall first use Proposition 2.2 to show that $\cha(u,v)$ is  approximately equal to the characteristic function of the joint distribution of $\re\widetilde{Z_H}$  and $\im\widetilde{Z_H}$, where
$$ \widetilde{Z_H}:= \frac{Z_H}{\sqrt{H/2}}.$$
 Then, using that $Z_H=X_1+\cdots+X_H$, where the $X_j$ are independent random variables uniformly distributed on the unit circle, we shall prove that the characteristic function of the joint distribution of $\re\widetilde{Z_H}$ and $\im\widetilde{Z_H}$ is close to the characteristic function of a  two-dimensional standard Gaussian distribution in a wide range, if $H$ is large. More precisely, we have
\begin{lem} Let $u,v$ be real numbers such that $|u|,|v|\leq H^{1/4}$. Then
$$\ex\left(e^{iu\re\widetilde{Z_H}+iv\im\widetilde{Z_H}}\right)=\exp\left(-\frac{u^2+v^2}{2}\right)\left(1+O\left(\frac{u^4+v^4}{H}\right)\right).$$
\end{lem}
\begin{proof} First, remark that
$$ \ex\left((\re X_1)^r(\im X_1)^s\right)= \frac{1}{2\pi}\int_{-\pi}^{\pi} (\cos \theta)^r(\sin \theta)^s d\theta=\frac{1}{2\pi}\int_{-\pi}^{\pi} (\sin t)^r(\cos t)^s dt,$$
by making the change of variable $t=\pi/2-\theta$ and using that $\sin\theta$ and $\cos\theta$ are periodic with period $2\pi$. Therefore since $\sin\theta$ is odd, then
\begin{equation}\ex\left((\re X_1)^r(\im X_1)^s)\right)=0,  \text{ if } r \text{ is odd or } s \text{ is odd}.
\end{equation}
Furthermore, the independence of the $X_j$ yields
\begin{equation}
\ex\left(e^{iu\re\widetilde{Z_H}+iv\im\widetilde{Z_H}}\right)= \ex\left(\exp\left(\frac{iu}{\sqrt{H/2}}\re X_1+\frac{iv}{\sqrt{H/2}}\im X_1\right)\right)^H.
\end{equation}
Now, using  (3.1) along with the fact that
$e^{ix}= 1+ix-x^2/2+(ix)^3/6+O(x^4)$ for all $x\in \mathbb{R}$, we deduce
\begin{align*}
&\ex\left(\exp\left(\frac{iu}{\sqrt{H/2}}\re X_1+\frac{iv}{\sqrt{H/2}}\im X_1\right)\right)\\
&= 1-\frac{u^2}{H}\ex\left((\re X_1)^2\right)-\frac{v^2}{H}\ex\left((\im X_1)^2\right)+ O\left(\frac{u^4+v^4+u^2v^2}{H^2}\right)\\
&= 1-\frac{u^2+v^2}{2H}+ O\left(\frac{u^4+v^4}{H^2}\right).
\end{align*}
since  $u^2v^2\leq u^4+v^4$. The result follows upon inserting this estimate in (3.2).
\end{proof}

\begin{proof}[Proof of Theorem 3.1]
Let $N$ be a positive integer such that $H^N\leq q^{1/20}.$ Using the Taylor expansion
$$ e^{ix}=\sum_{k=0}^{2N-1}\frac{(ix)^k}{k!}+ O\left(\frac{x^{2N}}{(2N)!}\right) \text{ for } x\in \mathbb{R},$$ we get
\begin{equation}
\cha(u,v)=\frac{1}{q}\sum_{x=0}^{q-1}\sum_{r=0}^{2N-1}\sum_{s=0}^{2N-1}
\frac{\left(iu\re\ns(x)\right)^r\left(iv\im\ns(x)\right)^s}{r!s!}+ E_1,
\end{equation}
where
\begin{equation}
\begin{aligned}
E_1&\ll \frac{1}{q}\sum_{x=0}^{q-1}\left(
\frac{\left(u\re\ns(x)\right)^{2N}}{(2N)!}+\frac{\left(v\im\ns(x)\right)^{2N}}{(2N)!}+
\frac{\left(uv\re\ns(x)\im\ns(x)\right)^{2N}}{(2N)!^2}\right)\\
&\ll \frac{u^{2N}2^N}{H^N(2N)!}M(2N,0)+ \frac{v^{2N}2^N}{H^N(2N)!}M(0,2N)+ \frac{(2uv)^{2N}}{H^{2N}(2N!)^2}M(2N,2N).\\
\end{aligned}
\end{equation}
Furthermore, combining Lemma 2.1 and Proposition 2.2, and noting that $B_k(H)\leq k!H^k$ we derive
\begin{equation}
M(2r,2s)\ll \ex\left((\re Z_H)^{2r}(\im Z_H)^{2s}\right)+ q^{-1/4}\ll (r+s)!H^{r+s},
\end{equation}
for any non-negative integers $r,s$ with  $H^{r+s}\leq q^{1/10}$. Hence, inserting this estimate in (3.4) and using that $(N!)^2\leq (2N)!$ we deduce
\begin{equation}
E_1\ll \frac{(2u^2)^N+(2v^2)^N}{N!}+ \frac{(2uv)^{2N}}{(2N)!}.
\end{equation}
Next we compute the main term on the RHS of (3.3). Define
\begin{align*}
\mathcal{F}_{N}(u,v)&=\frac{1}{q}\sum_{x=0}^{q-1}\sum_{r=0}^{2N-1}\sum_{s=0}^{2N-1}
\frac{\left(iu\re\ns(x)\right)^r\left(iv\im\ns(x)\right)^s}{r!s!}\\
&= \sum_{r=0}^{2N-1}\sum_{s=0}^{2N-1}\frac{(iu)^r(iv)^s}{(H/2)^{(r+s)/2}r!s!}M(r,s).
\end{align*}
Then, appealing to Proposition 2.2 we obtain
$$ \mathcal{F}_{N}(u,v)= \sum_{r=0}^{2N-1}\sum_{s=0}^{2N-1}\frac{(iu)^r(iv)^s}{r!s!}\ex\left(\left(\re\widetilde{Z_H}\right)^r\left(\im\widetilde{Z_H}\right)^s\right)+ O\left(q^{-1/4}(1+u^{2N})(1+v^{2N})\right).$$
Moreover, the main term on the RHS of the last estimate equals
\begin{equation*}
\begin{aligned}
&\ex\left(\left(\sum_{r=0}^{2N-1}\frac{(iu\re\widetilde{Z_H})^r}{r!}\right)\left(\sum_{s=0}^{2N-1}\frac{(iv\im\widetilde{Z_H})^s}{s!}\right)\right)\\
&= \ex\left(\left(e^{iu\re\widetilde{Z_H}}+O\left(\frac{(u\re\widetilde{Z_H})^{2N}}{(2N)!}\right)\right)\left(e^{iv\im\widetilde{Z_H}}+O\left(\frac{(v\im\widetilde{Z_H})^{2N}}{(2N)!}\right)\right)\right)\\
&= \ex\left(e^{iu\re\widetilde{Z_H}+iv\im\widetilde{Z_H}}\right) +O\left(\frac{(2u^2)^N+(2v^2)^N}{N!}+ \frac{(2uv)^{2N}}{(2N)!}\right),
\end{aligned}
\end{equation*}
which follows from (3.5). Thus, we infer from Lemma 3.2 that
\begin{equation*}
\begin{aligned}
\mathcal{F}_{N}(u,v)=&\exp\left(-\frac{u^2+v^2}{2}\right)\left(1+O\left(\frac{u^4+v^4}{H}\right)\right)\\
 &+ O\left(\frac{(2u^2)^N+(2v^2)^N}{N!}+ \frac{(2uv)^{2N}}{(2N)!}+ q^{-1/4}(1+u^{2N})(1+v^{2N})\right).
 \end{aligned}
\end{equation*}
Finally, combining this estimate with (3.3) and (3.6) completes the proof.
\end{proof}

\section{The distribution of $S_{\chi,H}(x)$: Proof of Theorem 1}

In order to prove Theorem 1 we shall appeal to the following Lemma of Selberg (Lemma 4.1 of \cite{Ts}), which provides a smooth approximation for the signum function. Selberg used this lemma in his proof that $\log\zeta(1/2+it)$ has a limiting two-dimensional Gaussian distribution (see \cite{Ts} and \cite{Se}).
\begin{lem}[Lemma 4.1 of \cite{Ts}]
Let $t>0$. Define
$$G(u)=\frac{2u}{\pi}+2(1-u)u\cot(\pi u) \quad \text{ for } u\in [0,1].$$
Then for all $x\in \mathbb{R}$ we have
$$ \textup{sgn}(x)=  \int_0^tG\left(\frac{u}{t}\right)\sin(2\pi i  ux)\frac{du}{u} + O\left(\left(\frac{\sin(\pi tx)}{\pi t x}\right)^2\right).$$
Moreover, $G(u)$ is differentiable and $0\leq G(u)\leq 2/\pi$ for  $u\in[0,1]$.
\end{lem}
Here and throughout we shall denote by $\mathbf{1}_{\alpha,\beta}$ the indicator function of the interval $[\alpha,\beta]$. Observe that
$$ \mathbf{1}_{\alpha,\beta}(x)=\frac{\textup{sgn}(x-\alpha)-\textup{sgn}(x-\beta)}{2}.$$ Moreover, define
$$ f_{\alpha,\beta}(u):=\frac{e^{-2\pi i \alpha u}-e^{-2\pi i \beta u}}{2}.$$
Then
\begin{equation}
|f_{\alpha,\beta}(u)|=\frac{1}{2}\left|\int_{2\pi \alpha u}^{2\pi \beta u} e^{-it}dt\right|\leq \pi u |\beta-\alpha|.
\end{equation}
Furthermore, it follows from Lemma 4.1 that
\begin{equation}
\mathbf{1}_{\alpha,\beta}(x)= \textup{Im} \int_0^tG\left(\frac{u}{t}\right)e^{2\pi i ux}f_{\alpha,\beta}(u)\frac{du}{u} + O\left(\frac{\sin^2(\pi t(x-\alpha))}{(\pi t (x-\alpha))^2}+ \frac{\sin^2(\pi t(x-\beta))}{(\pi t (x-\beta))^2}\right).
\end{equation}
Let $\mathcal{R}=[a,b]\times [c,d]$ be a rectangle in the complex plan, and denote by $\mathbf{1}_{\mathcal{R}}$ the indicator function of $\mathcal{R}$. Moreover, let $z=x+iy$ be a complex number. Then using the identity
\begin{equation}
\textup{Im}(w_1)\textup{Im}(w_2)= \frac{1}{2}\textup{Re}(w_1\overline{w_2}-w_1w_2),
\end{equation}
we infer from (4.2) that
\begin{equation}
\begin{aligned}
\mathbf{1}_{\mathcal{R}}(z)=& \frac12\textup{Re}\int_0^t\int_0^tG\left(\frac{u}{t}\right)G\left(\frac{v}{t}\right)\left(e^{2\pi i(ux-vy)}f_{a,b}(u)\overline{f_{c,d}(v)}-e^{2\pi i(ux+vy)}f_{a,b}(u)f_{c,d}(v)\right)\frac{du}{u}\frac{dv}{v}\\
&+ O\left(\frac{\sin^2(\pi t(x-a))}{(\pi t (x-a))^2}+ \frac{\sin^2(\pi t(x-b))}{(\pi t (x-b))^2}+\frac{\sin^2(\pi t(y-c))}{(\pi t (y-c))^2}+ \frac{\sin^2(\pi t(y-d))}{(\pi t (y-d))^2}\right).
\end{aligned}
\end{equation}
In order to prove Theorem 1 we shall require the following lemma:
\begin{lem} Let $t$ be a large positive real number. Then, uniformly for all real numbers $a<b$ we have
$$ \textup{Im}\int_{0}^tG\left(\frac{u}{t}\right)e^{-(2\pi u)^2/2}f_{a,b}(u)\frac{du}{u}
=\frac{1}{\sqrt{2\pi}}\int_a^b e^{-x^2/2}dx+O\left(\frac1t\right).$$
\end{lem}
\begin{proof}
Let $X$ be a standard Gaussian random variable. Since $\ex(e^{itX})=e^{-t^2/2}$, then
$$\textup{Im}\int_{0}^tG\left(\frac{u}{t}\right)e^{-(2\pi u)^2/2}f_{a,b}(u)\frac{du}{u}
= \ex\left(\textup{Im}\int_{0}^tG\left(\frac{u}{t}\right)e^{2\pi i u X}f_{a,b}(u)\frac{du}{u}\right).$$
On the other hand we have
$$\frac{1}{\sqrt{2\pi}}\int_a^b e^{-x^2/2}dx= P(X\in [a,b])=\ex(\mathbf{1}_{a,b}(X)).$$
Therefore, it follows from (4.2) that
\begin{align*}
&\textup{Im}\int_{0}^tG\left(\frac{u}{t}\right)e^{-(2\pi u)^2/2}f_{a,b}(u)\frac{du}{u}
- \frac{1}{\sqrt{2\pi}}\int_a^b e^{-x^2/2}dx\\
&\ll \ex\left(\frac{\sin^2(\pi t(X-a))}{(\pi t (X-a))^2}+ \frac{\sin^2(\pi t(X-b))}{(\pi t (X-b))^2}\right).
\end{align*}
To bound the RHS of the last inequality we use the following identity:
\begin{equation}
 \frac{\sin^2(\pi t x)}{(\pi tx)^2}= \frac{2(1-\cos(2\pi tx))}{t^2(2\pi x)^2}= \frac{2}{t^2}\int_0^t(t-v) \cos(2\pi xv)dv.
\end{equation}
Let $l$ be a real number. Then (4.5) yields
\begin{align*}
\ex\left(\frac{\sin^2(\pi t(X-l))}{(\pi t (X-l))^2}\right)&= \frac{2}{t^2}\textup{Re}\int_{0}^t(t-v)e^{-2\pi i l v}\ex\left(e^{2\pi i v X}\right)dv\\
&= \frac{2}{t^2}\int_{0}^t (t-v)\cos(2\pi l v)e^{-(2\pi v)^2/2}dv \ll \frac{1}{t}.
\end{align*}
This concludes the proof.
\end{proof}
We are now ready to prove Theorem 1.
\begin{proof}[Proof of Theorem 1]
Let $ N\leq \log q/(20\log H)$ be a positive integer and $1\leq t\leq \min(H^{1/4},N)$ be a real number to be chosen later. Then, using (4.4) we deduce that $\frac{1}{q}\sum_{x=0}^{q-1}\mathbf{1}_{\mathcal{R}}(\ns(x))$ equals
\begin{equation}
\begin{aligned}
&\frac12\textup{Re}\int_0^t\int_0^tG\left(\frac{u}{t}\right)G\left(\frac{v}{t}\right)\left(\cha(2\pi u, -2\pi v)f_{a,b}(u)\overline{f_{c,d}(v)}-\cha(2\pi u, 2\pi v)f_{a,b}(u)f_{c,d}(v)\right)\frac{du}{u}\frac{dv}{v}\\
&+ O\Big(I_{\chi}(t,a)+ I_{\chi}(t,b)+J_{\chi}(t,c)+J_{\chi}(t,d)\Big),
\end{aligned}
\end{equation}
where
$$I_{\chi}(t,l)= \frac{1}{q}\sum_{x=0}^{q-1}\frac{\sin^2\big(\pi t(\textup{Re}\ns(x)-l)\big)}{(\pi t(\textup{Re}\ns(x)-l))^2},$$
and
$$ J_{\chi}(t,l)= \frac{1}{q}\sum_{x=0}^{q-1}\frac{\sin^2\big(\pi t(\textup{Im}\ns(x)-l)\big)}{(\pi t(\textup{Im}\ns(x)-l))^2}.$$
We begin by estimating the main term in (4.6). First note that (4.1) implies
$$
 \frac{|f_{a,b}(u)f_{c,d}(v)|}{uv}\leq \pi^2\mu_2(\mathcal{R}).
$$
Therefore, using this inequality and appealing to Theorem 3.1 we find that the integral in (4.6) equals
\begin{equation}
\begin{aligned}
&\frac12\textup{Re}\int_0^t\int_0^tG\left(\frac{u}{t}\right)G\left(\frac{v}{t}\right)
\exp\left(-\frac{(2\pi u)^2+(2\pi v)^2}{2}\right)\left(f_{a,b}(u)\overline{f_{c,d}(v)}-f_{a,b}(u)f_{c,d}(v)\right)\frac{du}{u}\frac{dv}{v}\\
&+ O\left(\mu_2(\mathcal{R})\left(\frac{1}{H}+ \frac{(4\pi t)^{2N+2}}{N!}+ \frac{(4\pi t)^{4N+2}}{(2N)!}+q^{-1/4}(4\pi t)^{4N+2}\right)\right).
\end{aligned}
\end{equation}
Now, using (4.3) along with Lemma 4.2 we deduce that the main term of (4.7) equals
\begin{equation}
\begin{aligned}
&\int_0^t\int_0^tG\left(\frac{u}{t}\right)G\left(\frac{v}{t}\right)
\exp\left(-\frac{(2\pi u)^2+(2\pi v)^2}{2}\right)\textup{Im}f_{a,b}(u)\textup{Im}{f_{c,d}(v)}\frac{du}{u}\frac{dv}{v}\\
&= \frac{1}{2\pi}\int_a^b\int_c^d\exp\left(-\frac{u^2+v^2}{2}\right)dudv + O\left(\frac1t\right).
\end{aligned}
\end{equation}
Next, we bound the contribution of the error term in (4.6). Using the identity (4.5) we derive
\begin{equation*}
\begin{aligned}
I_{\chi}(t,l)&= \frac{1}{q}\sum_{x=0}^{q-1}\int_0^t \frac{2(t-v)}{t^2}\cos\Big(2\pi v\big(\textup{Re}\ns(x)-l\big)\Big)dv\\
&= \frac{1}{q}\textup{Re}\sum_{x=0}^{q-1}\int_0^t \frac{2(t-v)}{t^2}e^{-2\pi ivl} \exp\Big(2\pi iv\textup{Re}\ns(x)\Big)dv\\
&= \textup{Re}\int_0^t \frac{2(t-v)}{t^2}e^{-2\pi ivl} \cha(2\pi v,0)dv.
\end{aligned}
\end{equation*}
Hence, it follows from Theorem 3.1 that
\begin{equation}
\begin{aligned}
I_{\chi}(t,l)&\ll \frac{1}{t}\int_{0}^t e^{-(2\pi v)^2/2}\left(1+O\left(\frac{v^4}{H}\right)\right)dv + (4\pi t)^{2N}\left(\frac{1}{N!}+q^{-1/4}\right)\\
&\ll \frac{1}{t}+ (4\pi t)^{2N}\left(\frac{1}{N!}+q^{-1/4}\right).
\end{aligned}
\end{equation}
A similar bound for $J_{\chi}(t,l)$ can be obtained along the same lines. Therefore, combining the estimates (4.6)-(4.9) we deduce
\begin{equation*}
\frac{1}{q} \sum_{x=0}^{q-1}\mathbf{1}_{\mathcal{R}}(\ns(x))= \frac{1}{2\pi}\iint_{\mathcal{R}}\exp\left(-\frac{x^2+y^2}{2}\right)dxdy +O(E_2),
\end{equation*}
where
$$E_2=\left(\mu_2(\mathcal{R})+1\right)\left(\frac{1}{t}+ \frac{(4\pi t)^{2N+2}}{N!}+ \frac{(4\pi t)^{4N+2}}{(2N)!}+q^{-1/4}(4\pi t)^{4N+2}\right).
$$
We choose $$ t= \min \left(H^{1/4}, \frac{1}{60\pi}\sqrt{\frac{\log q}{\log H}}\right),$$ and  $N=[(8\pi t)^2].$ Hence, Stirling's formula yields
$$ \frac{(4\pi t)^{2N+2}}{N!}+ \frac{(4\pi t)^{4N+2}}{(2N)!}\ll \frac{1}{t^2},$$
which implies
$$E_2 \ll \left(\mu_2(\mathcal{R})+1\right)\left(H^{-1/4}+ \sqrt{\frac{\log H}{\log q}}+ q^{-1/4}(4\pi t)^{(20\pi t)^2}\right).$$
Finally, noting that 
$$ q^{-1/4}(4\pi t)^{(20\pi t)^2}\leq q^{-1/4}(4\pi H)^{\log q/(9\log H)}\ll q^{-1/8},$$
concludes the proof.
\end{proof}


\begin{thebibliography}{DDDD}

\bibitem[1]{CS} S. Chatterjee and K. Soundararajan,
\emph{Random multiplicative functions in short intervals},
To appear in Int. Math. Res. Not.


\bibitem[2] {DE} H. Davenport and P. Erd\"os,
\emph{The distribution of quadratic and higher residues},
Publ. Math. Debrecen 2, (1952). 252–-265.


\bibitem[3] {IK} H. Iwaniec and E. Kowalski,
\emph{Analytic number theory},
American Mathematical Society Colloquium Publications, 53.
American Mathematical Society, Providence, RI, 2004.

\bibitem[4]{MZ}  K. H. Mak and A. Zaharescu,
\emph{The Distribution of Values of Short Hybrid Exponential Sums on Curves over Finite Fields},
Math. Res. Lett. 18 (2011), no. 1, 155–-174.


\bibitem[5]{Ng} N. Ng,
 \emph{The M\"obius function in short intervals},
  Anatomy of Integers,  CRM Proceedings and Lecture Notes, Volume 46, 2008, 247--258.

\bibitem[6]{Ts}  K. M. Tsang,
\emph{The distribution of the values of the zeta function},
Thesis, Princeton
University, October 1984, 179 pp.

\bibitem[7] {Se} A. Selberg,
\emph{Old and new conjectures and results about a class of Dirichlet series},
Proceedings of the Amalfi Conference on Analytic Number Theory (Maiori, 1989) Univ. Salerno, Salerno  (1992),  367--385.


\bibitem[8]{We} A. Weil,
\emph{On some exponential sums},
Proc. Nat. Acad. Sci. U. S. A. 34, (1948). 204–-207.


\end{thebibliography}
\end{document}